\documentclass[12pt,letterpaper]{amsart}
\usepackage{amsmath, amsthm, amssymb, xspace, mathrsfs}
\usepackage[tmargin=1.4in,bmargin=1.2in,rmargin=1.4in,lmargin=1.4in]{geometry}
\usepackage[breaklinks=true]{hyperref}%To get clickable links to papers
\usepackage{amsmath,amscd}
\usepackage{tikz-cd}

\theoremstyle{plain}
\newtheorem{theorem}{Theorem}[section]
\newtheorem{lemma}{Lemma}[section]

\newtheorem{problem}{Problem}[section]

\theoremstyle{definition}

\newcommand{\comment}[1]{}

\DeclareMathOperator{\Frac}{\ensuremath{Frac}}

\begin{document}
\title[The noncommutative Noether's problem]{The noncommutative Noether's problem is almost equivalent to the classical Noether's problem}
\author{Akaki Tikaradze}
\email{ Akaki.Tikaradze@utoledo.edu}
\address{University of Toledo, Department of Mathematics \& Statistics, 
Toledo, OH 43606, USA}
\begin{abstract}

Motivated by the classical Noether's problem, J. Alev  and F. Dumas proposed the following question, commonly referred to as the noncommutative
Noether's problem: Let a finite group $G$ act linearly on $\mathbb{C}^n$
 inducing the action on $\Frac(A_n(\mathbb{C}))$-the skew field of fractions of the $n$-th Weyl algebra $A_n(\mathbb{C}),$ is $\Frac(A_n(\mathbb{C}))^G$ isomorphic to $\Frac(A_n(\mathbb{C}))?$ In this note we show that if $\Frac(A_n(\mathbb{C})^{G}\cong \Frac(A_n(\mathbb{C}))$ then for any algebraically closed field $\bf{k}$ of large enough characterisitic, field $\bold{k}(x_1,\cdots, x_n)^G$ is stably rational. This result allows us to produce counterexamples to
 the noncommutative Noether's problem based on well-known counterexamples to the Noether's problem for algebraically closed fields.

\end{abstract}

\maketitle

Let $F$ be a field. Let a finite group $G$ act on variables $x_1,\cdots, x_n$ by permutations. Then the classical Noether's problem
asks whether the fixed field $F(x_1,\cdots, x_n)^G$ is a purely transcendental extension of $F,$ or
equivalently if $$F(x_1,\cdots, x_n)^G\cong F(x_1,\cdots, x_n).$$
 The first instances when
the Noether's problem (over $\mathbb{Q}$) has a negative answer was demonstrated by R. Swan \cite{Sw}. Subsequently Saltman constructed
counterexamples over algebraically closed fields \cite{S}. More specifically, Saltman constructed infinite family of $p$-groups $G$, such that for any algebraically closed field
of characteristic not dividing $|G|$, the fixed field $F(x(g), g\in G)^G$ is not stably rational: no purely transcendental
extension of  $F(x(g), g\in G)^G$ is purely transcendental over $F.$ In fact to the best of our knowledge all known counterexamples to Noether's problem
have the property that the fixed field is not stably rational.

   Motivated by the Noether's problem, Alev and Dumas \cite{AD} considered its noncommutative analogue, which is now commonly referred to
   as the noncommutative Noether's problem.  Throughout given a noetherian domain $A$, by $\Frac(A)$ we denote its skew field of fractions.

 \begin{problem}[Noncommutative Noether's problem]
 Let $G\subset GL_n(\mathbb{C})$ be a finite subgroup, then $G$ acts on the $n$-th Weyl algebra $A_n(\mathbb{C})$, hence acting
 on its skew field of fractions $\Frac(A_n(\mathbb{C}))$. Is $\Frac(A_n(\mathbb{C}))^G$ isomorphic to $\Frac(A_n(\mathbb{C}))$?
 \end{problem}

There is a clear analogy between the noncommutative Noetherian problem and the Gelfand-Kirillov conjecture about fraction fields
of enveloping algebras.
Noncommutative Noether's problem has been of interest to ring theorists for some time now. It was shown in \cite{EFOS} that
the answer is positive for complex reflection groups. More generally, it was proved by Futorny and Schwarz in \cite{FS} that if the classical Noether's problem has 
a positive answer for $G$, then so does the noncommutative one.

In this note we show that counterexamples to the Noether's problem by Saltman also give counterexamples to the noncommutative
one. 
\begin{theorem}\label{main}
Let $G\subset GL_n(\mathbb{Z})$ be a finite group such that $\Frac(A_n(\mathbb{C}))^G\cong \Frac(A_n(\mathbb{C}))$.
Then for all sufficiently large primes $p$ field $\overline{\mathbb{F}_p}(x_1,\cdots, x_n)^G$ is stably rational.
In particular, if $G$ is any of the groups constructed by Saltman \cite{S}, then viewing 
$A_n(\mathbb{C})$ as the ring of differential operators
on $\mathbb{C}[x(g), g\in G]$ with the usual action of $G$, we have that $\Frac(A_n(\mathbb{C}))^G\ncong \Frac(A_n(\mathbb{C})).$
\end{theorem}

Our result and its proof is motivated by Premet's work on the Gelfand-Kirillov conjecture for semi-simple Lie algebras \cite{P}, in which
he settled the conjecture negatively for simple Lie algebras of types  $B_n, n\geq 3, D_n, n\geq 4, E_6, E_7, E_8, F_4.$

Recall that if $char(F)=p,$ then the center of the $n$-th Weyl algebra 
$$A_n(F)=F\langle x_1,\cdots, x_n, y_1,\cdots, y_n\rangle$$
which is given with the usual generators and the usual relations $$[y_i, x_j]=\delta_{ij},\quad  [y_i, y_j]=0=[x_i, x_j]$$
 equals to $$Z(A_n(F))=F[x_1^p,\cdots, x_n^p, y_1^p,\cdots, y_n^p].$$
More generally, given a smooth affine variety $X$ over an algebraically closed field $\bf{k}$
of characteristic $p,$ we recall
 the following well-known description of the center of the ring of crystalline differential operators
$D(X)$(no divided powers). Given a  vector field $\xi\in Der_{\bold{k}}(\mathcal{O}(X)),$ denote by 
$\xi^{[p]}\in Der_{\bold{k}}(\mathcal{O}(X))$ its $p$-th power. Then $Z(D(X))$ is generated over $\mathcal{O}(X)^p$
by $$\xi^p-\xi, \xi\in Der_{\bold{k}}(\mathcal{O}(X)),$$ which yields a canonical isomorphism between $Z(D(X))$ and
the Frobenius twist of the ring of functions on the cotangent bundle of $X,$ see \cite{BMR}.  In particular, it follows that
$\Frac(Z(D(X)))$ is a purely transcendental extension of $\bold{k}(X)^p.$

The next result allows the passage from characteristic 0 to characteristic $p$ setting, it is very similar to [\cite{P}, Theorem 2.1].

\begin{lemma}\label{reduction}

Let $S\subset\mathbb{C}$ be a finitely generated ring, let $A$ be an $S$-algebra, such that it can be equipped
with an ascending filtration concentrated in non negative degrees, so that $gr(A)$ is a finitely generated commutative domain over $S.$ 
Assume that $gr(A)\otimes\mathbb{C}$ is a domain. Let $\Frac(A\otimes\mathbb{C})\cong \Frac(A_n(\mathbb{C}))$ for some $n.$
Then for all $p\gg 0$, there exists a base change $S\to\bf{k}$ to an algebraically closed field of characteristic $p$ so that 
$\Frac(A_{\bf{k}})\cong\Frac(A_n(\bf{k})).$

\end{lemma}
\begin{proof}
The proof closely follows that of [\cite{P}, Theorem 2.1], so some details are omitted.
Let $u_i, u'_i, v_i, v'_i\in A, 1\leq i\leq n$ be so that $u_i^{-1}u'_i, v_i^{-1}v'_i, 1\leq i\leq n$
generate $\Frac(A)$ and  $$[u_i^{-1}u'_i, v_i^{-1}v'_i]=\delta_{ij}.$$
After getting rid of all denominators, we may construct nonzero elements $y_1,\cdots, y_l\in A$ and noncommutative polynomials
$\phi_j,1\leq j\leq k$ such that
the above statements are equivalent to $$\phi_j(y_1,\cdots, y_l)=0, 1\leq j\leq k.$$
Now it easily follows that for any large enough prime $p$ there exists a base change $S\to \bf{k}$ to an algebraically closed field
of characteristic $p,$ such that $gr(A)\otimes_S\bf{k}$ and hence $A_{\bf{k}}$ is a domain and images of $u_i^{-1}u'_i, v_i^{-1}v'_i, 1\leq i\leq n$
satisfy relations of the Weyl algebra and generate $\Frac(A_{\bf{k}}).$
\end{proof}

We also need the following result.

\begin{lemma}\label{trivial}

Let $G\subset GL_n(F)$ be a finite group, $F$ algebraically closed field of characteristic $p.$
Then $Z(\Frac(A_n(F))^G)=Z(\Frac(A_n(F)))^G$ is a purely transcendental extension on $F(x_1^p,\cdots, x_n^p)^{G}.$

\end{lemma}
\begin{proof}
Since $G$ acts on $\Frac(A_n(F))$ by outer automorphisms, it follows from standard results about fixed rings (see for example [\cite{M}, Corollary 6.17])
 that
$$Z(\Frac(A_n(F))^G)=(Z(Frac(A_n(F)))^G.$$ 
Let $h^{reg}\subset F^n$ denote the open locus of points with trivial stabilizer in $G.$
Thus, $\Frac(A_n(F))=Frac(D(h^{reg}))$, so $\Frac(A_n(F))^G=\Frac(D(h^{reg})^G).$ Also, since $G$ acts freely on $h^{reg}$, 
we have $D(h^{reg})^G=D(h^{reg}/G).$
Hence, $$Z(\Frac(A_n(F))^G)=\Frac(Z(D(h^{reg}/G))$$ is a purely transcendental extension of
$$F(h^{reg}/G)^p=F(x_1^p,\cdots, x_n^p)^{G}.$$
%Let $U\subset h^{reg}/G$ be an open subset on which the tangent bundle is trivial. Let $t_1,\cdots, t_n$ denote the generating vector fields.
%Then we get that $\Frac(Z(D(U))=\Frac(D(h^{reg}/G))$ is a purely transcendental extension of 
%$\Frac(\mathcal{O}(U)^p)=F(x_1^p,\cdots, x_n^p)^{G}$ generated by $t_1^p,\cdots, t_n^p.$
\end{proof}

\begin{proof}[Proof of Theorem \ref{main}]

Assume that  $\Frac(A_n(\mathbb{C}))^G\cong \Frac(A_n(\mathbb{C})).$
It follows from Lemma \ref{reduction} that for all $p\gg 0$ and algebraically closed field $F$ of characteristic $p$, we have 
$$Z(\Frac(A_n(F))^G)\cong Z(\Frac(A_n(F)))=F(x_i^p, y_i^p, 1\leq i\leq n).$$ But by Lemma \ref{trivial}  $Z(\Frac(A_n(F)^G)$
is a purely transcendental extension of $F(x_i^p, 1\leq i\leq n)^G.$
Hence we conclude that $F(x_i, 1\leq i\leq n)^G$ is stably rational.
\end{proof}


\begin{thebibliography}{}

\bibitem[AD]{AD}
J.~Alev, F.~Dumas, {\em Operateurs differentiels invariants et probleme de Noether}. In: Bernstein, J., Hinich, V., Melnikov, A. (eds.) Studies in Lie Theory. Birkhauser, Boston (2006) 

\bibitem[BMR]{BMR}
R.~Bezrukavnikov, I.~Mirkovic,  D.~Rumynin, 
{\em Localization of modules for a semisimple Lie algebra in prime characteristic},
Ann. of Math. (2) 167 (2008), no. 3, 945--991.

\bibitem[EFOS]{EFOS}
F.~Eshmatov, V.~Futorny, S.~Ovsienko, J.~Schwarz,
{\em Noncommutative Noether's problem for complex reflection groups},
Proc. Amer. Math. Soc. 145 (2017), no. 12, 5043--5052.

\bibitem[FS]{FS}
V.~Futorny, J.~Schwarz,
{\em Noncommutative Noether's problem vs classic Noether's problem},
Math. Z. 295 (2020), no. 3-4, 1323--1335.


\bibitem[M]{M}

S.~Montgomery, {\em Fixed rings of finite automorphism groups of associative rings}, (1980) Lecture Notes in Math.



\bibitem[P]{P}
A.~Premet, {\em Modular Lie algebras and the Gelfand-Kirillov conjecture},
Invent. Math. 181 (2010), no. 2, 395--420.

\bibitem[S]{S}
D.~Saltman,
{\em Noether's problem over an algebraically closed field},
Invent. Math. 77 (1984), no. 1, 71--84.

\bibitem[Sw]{Sw}
R.~Swan,
{\em Invariant rational functions and a problem of Steenrod},
Invent. Math. 7 (1969), 148--158. 


%{https://mathoverflow.net/questions/350590/lifting-g-invariants-from-characteristic-p-gg-0-to-characteristic-0-for-a-re}


\end{thebibliography}
\end{document}